%% file: kinetic_trace_estimate_arxiv.tex
\documentclass[11pt,a4paper]{article}
\usepackage[T1]{fontenc}
\usepackage[utf8]{inputenc}
\usepackage{lmodern}
\usepackage{amsmath,amssymb,amsthm,mathtools}
\usepackage{geometry}
\usepackage{xcolor}
\usepackage{hyperref}
\usepackage{tikz} 
\usetikzlibrary{calc}

\newcommand{\RR}{\mathbb{R}}

\newcommand{\R}{\mathcal{R}}

\newcommand{\Sd}{\mathbb{S}^{d-1}}
\newcommand{\Rd}{\RR^d}
\newcommand{\A}{\mathcal{A}}

\newcommand{\gradR}{\nabla_r}
\newcommand{\dd}{\mathrm{d}}

\newcommand{\moll}{\mu}
\newcommand{\Gammain}{\Gamma_{\mathrm{in}}}

\def\hyp{\mathrm{hyp}}
\newcommand{\Hyp}{H^1_{\hyp}(\R;\Sd)}

\DeclareMathOperator{\dist}{dist}

\theoremstyle{plain}
\newtheorem{lemma}{Lemma}[section]
\newtheorem{theorem}{Theorem}
\newtheorem{proposition}[lemma]{Proposition}
\theoremstyle{definition}
\newtheorem{remark}[lemma]{Remark}

\title{A trace theorem for 
spherical kinetic Sobolev spaces \\on $\rho$-convex domains}
\author{%
    Herbert Egger%
    \thanks{Institute for Numerical Mathematics, Johannes Kepler University Linz and Johann Radon Institute for Computational and Applied Mathematics, Altenbergerstr.~69, 4040 Linz, Austria.  
    {\small\textit{e-mail}:} {\small \texttt{herbert.egger@jku.at}}}
    \and
    Matthias Schlottbom%
    \thanks{Department of Applied Mathematics, University of Twente, P.O. Box 217, 7500 AE Enschede, The Netherlands.  
    {\small\textit{e-mail}:} {\small \texttt{m.schlottbom@utwente.nl}}}}
\date{}

\begin{document}
\maketitle

\begin{abstract}
Trace theorems are an indispensable tool for the analysis of kinetic equations. They have been established in wide generality by Cessenat and co-workers for radiative transfer and related applications. 
A variety of trace estimates have been established for the kinetic Fokker-Planck and Kolmogorov equation, typically requiring smoothness of the underlying domain; see the recent survey by Niebel \& Valentini.
In this work, we prove a new trace estimate for kinetic Sobolev spaces over the sphere for $\rho$-convex domains which, in general, may have a non-smooth boundary. Similar to the work of Cessenat, we use characteristics to obtain trace estimates in weighted trace spaces with explicit constants. For completeness, we also present a density result for the corresponding function spaces on Lipschitz domains. 
\end{abstract}

\section{Introduction}
\label{sec:introduction}
Kinetic Fokker--Planck equations arise as forward-peaked approximations of Boltzmann transport and have important applications in particle transport and radiotherapy, including dose calculation and treatment planning \cite{Bedford2023,Hensel2006}. 
In MRI-guided radiotherapy, the transport of charged particles in strong magnetic fields provides further motivation for such models \cite{dePooter2021,StAubin2016}.
Lions' representation theorem, see, e.g. \cite{Arendt2023},
 has been used to establish existence for kinetic Fokker--Planck equations,
 see, e.g., \cite{BalPalacios2020,Carrillo1998,Herty2012}.
The corresponding uniqueness arguments, however, require the existence of traces for functions in kinetic Sobolev spaces and
ultimately rely on the assertion that smooth functions vanishing on the outflow boundary are dense in the associated graph space, 
which remains an open problem; see \cite[Appendix~A]{ArmstrongMourrat2021}.
A rigorous variational treatment of the corresponding boundary value problems therefore requires an appropriate trace theory for the underlying kinetic Sobolev spaces, which is the focus of this work.

\subsection*{Kinetic Sobolev space over the sphere}
Let $\R\subset\Rd$ be a bounded domain and $\Sd\subset\Rd$ the unit sphere in dimension $d \ge 2$. 
A natural function space arising in the study of kinetic Fokker--Planck equations and related models is the spherical kinetic Sobolev spaces~\cite{AlbrittonEtAl24,Valentini2026}
\begin{align*}
    \Hyp=\{v\in L^2(\R;H^1(\Sd)): s\cdot\gradR v\in L^2(\R;H^{-1}(\Sd))\}.
\end{align*}
Here $H^1(\Sd)$ is the space of square-integrable functions on $\Sd$ with square-integrable weak gradient, 
$H^{-1}(\Sd)=H^1(\Sd)^*$ is its dual, and $L^2(\R;X)$ denotes the Bochner space of square-integrable functions from $\R$ to some Hilbert space $X$. The graph norm, defined by 
\begin{align*}
    \|v\|_{\hyp}^2=\|v\|_{L^2(\R;H^1(\Sd))}^2+\|s\cdot\gradR v\|_{L^2(\R;H^{-1}(\Sd))}^2,
\end{align*}
together with the associated inner product,
renders $\Hyp$ a Hilbert space. Variational methods for kinetic equations based on such spaces have been studied, e.g., in \cite{AlbrittonEtAl24,Auscher2024,BES25,BrunkenSmetana22}. 

\subsection{A kinetic trace theorem for \texorpdfstring{$\rho$}{rho}-convex domains}

Trace estimates for functions in $\Hyp$ are a subtle issue, and usually, a certain regularity of the domain $\R$ is required; see \cite{Mischler00,NV26}. 
Below, we provide a trace theorem that applies to a specific class of potentially non-smooth domains. 
Recall that a domain $\R \subset \RR^d$ is \emph{$\rho$-convex}, if there exists a radius $\rho>0$ and a set $\A\subset\Rd$ such that
\begin{align}     \label{eq:r-convex}
    \overline{\R}=\bigcap_{a\in \A}\overline{B(a,\rho)}.
\end{align}
Here $B(a,\rho)$ is the open ball of radius $\rho$ centered at $a$.
Note that $\rho$-convex domains are convex and bounded, but may have non-smooth boundaries; e.g., the intersection of two balls. 
In particular, the unit outward normal vector $n(r)$ exists for $\mathcal{H}^{d-1}$-almost every $r\in\partial \R$.
We refer to \cite{FrankowskaOlech81,NacryThibault24} for further details on $\rho$-convexity and related notions of convexity. By 
\begin{align}
    \tau_\pm(r,s)=\inf\{t>0: r \pm ts\notin\R\}
    \label{eq:taup}
\end{align}
we denote the \emph{exit time}, i.e., the distance of the point $r$ to the boundary $\partial \R$, in direction $\pm s$. Correspondingly, $\tau(r,s) = \tau_+(r,s) + \tau_-(r,s)$ denotes the length of the segment of the line $r + t s$ that lies inside $\overline \R$; 
\begin{figure}[ht!]
\centering
\input{r-convex.tikz}
\caption{Typical examples of $\rho$-convex domains: the ball $B(a,\rho)$ and the Reuleaux triangle, which is the intersection of three balls $B(a_1,\rho) \cap B(a_2,\rho) \cap B(a_3,\rho)$ of radius $\rho$ centered at the vertices of a triangle. The left plot further illustrates the exit times $\tau_\pm$ and the exit points $r_\pm$ for the line $r+t s$ passing through $r$ in direction $s$.}
\label{fig:r-convex}
\end{figure}
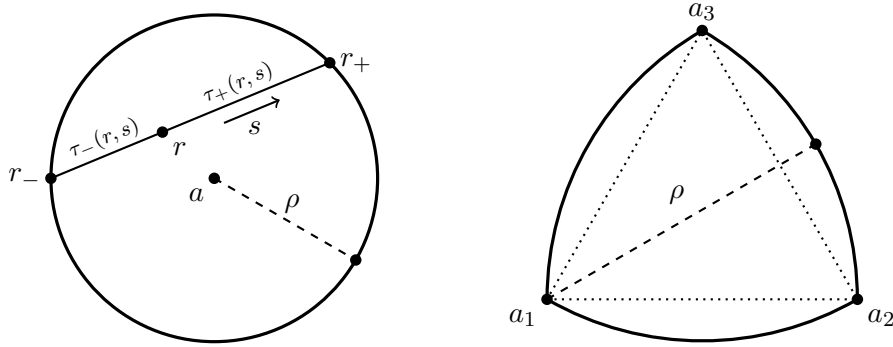
see \cite[Sec.~2.1]{Agoshkov1998} or \cite[Ch.~XXI, App. §~2] {DautrayLions6} and Figure~\ref{fig:r-convex} for illustration. 
As the main contribution of this note, we will establish the following result. 
\begin{theorem}[Trace theorem for $\Hyp$]\label{thm:trace} $ $\\
Let $\R$ be a $\rho$-convex domain. 
Then for all $v\in C^1(\overline\R\times\Sd)$, there holds 
\begin{align}\label{eq:trace}
\int_{\partial\R \times \Sd} |v(r,s)|^2 \, \tau(r,s) \, |s\cdot n(r)| \, \dd(r,s) 
\leq 2 \,(1+4\rho) \, \|v\|_{\hyp}^2.
\end{align}
By continuity and density, the trace mapping $\gamma : C^1(\overline\R\times\Sd) \to L^2(\partial\R \times \Sd;\tau|s\cdot n|)$ can be extended as a linear continuous operator to $\Hyp$ and  \eqref{eq:trace} remains valid.
\end{theorem}
A detailed proof of the trace estimate \eqref{eq:trace} will be presented in Section~\ref{sec:trace}. The density result underlying the final claim will be stated below and proven in Section~\ref{sec:density}.

\begin{remark}
The weighted norm on the left hand side of \eqref{eq:trace} also arises in trace estimates for the radiative transfer equation~\cite{CessenatDautray1984}. For bounded Lipschitz domains $\R$, there holds
\begin{align}\label{eq:trace_rte}
\int_{\partial\R \times \Sd} |v(r,s)|^q \, \tau(r,s) \, |s\cdot n(r)| \, \dd(r,s) \leq C_q \, \|v\|_{W^q}^q
\end{align}
where $\|v\|_{W^q}^q= \|v\|_{L^q(\R \times \Sd)}^q + \|s \cdot \nabla v\|_{L^q(\R \times\Sd)}^q$ for $q \ge 1$, defines the graph norm of the natural function spaces arising in the analysis of radiative transfer~\cite{Agoshkov1998,DautrayLions6,ManteuffelResselStarke00}. 
The weight $\tau$ on the left hand side of this estimate can be removed by switching to weighted norms on the right hand side; see \cite{Boulanouar2009,Egger2013} for details.
In \cite{NV26}, trace estimates for kinetic Sobolev spaces of the form 
\begin{align}\label{eq:trace_nv}
\int_{\partial\R \times \Sd} |v(r,s)|^2 \,  |s\cdot n(r)|^p \, \dd(r,s) \leq C \, \|v\|_{\hyp}^2, \qquad p \ge 1,
\end{align}
have been established for domains $\R$ with $C^{1,1/(1+p)}$-regular boundary. 
Counterexamples illustrate that such estimates do, in general, not hold for less regular domains.
We will show that for $\rho$-convex domains $\tau |s \cdot n| \le c \, |s \cdot n|^2$, so our estimate \eqref{eq:trace} has a smaller weight than \eqref{eq:trace_nv} for $p \le 2$. 
Related trace estimates for velocity spaces other than $\Sd$ can be found in \cite{AlbrittonEtAl24,AvelinHou25,Mischler00,Silvestre22}. 
For a detailed account of trace theorems for kinetic equations, we refer to \cite[Sec.~1.7]{NV26}.
\end{remark}

\subsection{A density theorem}
To fully justify the second claim of Theorem~\ref{thm:trace}, we will establish the density of smooth functions in $\Hyp$ for Lipschitz domains $\R$ which slightly generalizes previous results. 
\begin{theorem}[Density in $\Hyp$]
\label{thm:density} $ $\\
Let $\R\subset\RR^d$ be a bounded Lipschitz domain. 
Then
$C^\infty(\overline\R\times\Sd)$ is dense in $\Hyp$.
\end{theorem}
A detailed proof of this result will be presented in Section~\ref{sec:density}. We note that $\rho$-convex domains are bounded and Lipschitz, and thus covered by the previous result. 
\begin{remark}
Density of smooth functions in kinetic Sobolev spaces has been established for different
velocity models. 
For an unrestricted velocity model, \cite[Prop.~2.2]{AlbrittonEtAl24} proves global density on bounded
$C^1$-domains; as noted in the proof of \cite[Prop.~6.1]{AlbrittonEtAl24}, the arguments used in the proof actually carry over to Lipschitz domains. \cite[Lemma~4.2]{Silvestre22} provides a weaker local variant on $C^{1,1}$-domains and
\cite[Lemma~5.2]{AvelinHou25} establishes a global result on $C^{1,1}$-domains by translated mollification.
\cite[Prop.~3.1]{BrunkenSmetana22} and \cite[Prop.~7.5]{NV26} establish 
density of smooth functions for the spherical velocity model for $C^1$-domains. 
Using the localization arguments mentioned in \cite[Sec.~6]{AlbrittonEtAl24}, these results can again be extended to Lipschitz domains. 
For completeness, we here present a detailed statement and proof of this result. 
A different strategy altogether, avoiding an independent density theorem, is to \emph{define}
the graph space as the completion of smooth functions, as done by \cite{ShengHan2013} for the
spherical model; this makes density trivial by construction but leaves the space's relation to
the transport graph space $\Hyp$ defined above implicit. 
\end{remark}

\section{Trace estimate for \texorpdfstring{$\rho$}{rho}-convex domains}
\label{sec:trace}

Recall that $\R\subset\Rd$ is a bounded domain and $\Sd\subset\Rd$ is the unit sphere. A key ingredient for our analysis is the fact that the exit times
\begin{align*}
   \tau_\pm(r,s)=\inf\{t>0: r \pm ts\notin\R\}
\end{align*}
depend smoothly on the direction $s$ if the domain $\R$ has sufficient curvature.  
A special case of the following result has been established in \cite{Corr25}. For convenience of the reader and later reference, we present the full statement and provide a self-contained proof. 
\begin{lemma} \label{lem:tau-p}
Let $\R \subset \Rd$ be a $\rho$-convex domain.
Then $\tau_\pm(r,s) \le 2 \rho$ and 
\begin{align} \label{eq:lipschitz}
|\tau_\pm(r,s) - \tau_\pm(r,s')| \le 2 \rho \, |s - s'| \quad \forall s,s' \in \Sd, \ r \in \Rd.
\end{align}
By Rademacher's theorem, this implies $|\nabla_s \tau_\pm(r,s)| \le 2 \rho$ for a.a. $r \in \Rd$ and $s \in \Sd$.
\end{lemma}
\begin{proof}
It suffices to consider the case $\tau_+$; since $\tau_-(r,s)=\tau_+(r,-s)$, the results for $\tau_-$ follow in the same manner.
The domain $\R$ is contained in at least one ball $B(a,\rho)$, see~\eqref{eq:r-convex}, thus the exit time is bounded by the diameter of $B(a,\rho)$, which already shows $\tau_+(r,s) \le 2 \rho$. 
As a next ingredient, we observe that the exit time for a ball $B(a,\rho)$ can be expressed explicitly by 
\begin{align*}
\tau_{+,a}(r,s) = \sqrt{\rho^2-|r-a|^2+\langle r-a,s\rangle^2
}-\langle r-a,s\rangle;
\end{align*}
see Figure~\ref{fig:exit-time} for an illustration. 
\begin{figure}[ht!]
\centering
\input{exit-time.tikz}
\caption{
Left:
Geometric construction of $\tau_{+,a}$ 
based on the auxiliary quantities and elementary formulas $y=\langle r-a,s\rangle$, $x^2+y^2=|r-a|^2$, and $(\tau_{+,a}+y)^2+x^2=\rho^2$.
Right: Sketch of a star-shaped domain (solid) and its dilation (dotted); see Proposition~\ref{pro:density}.
}
\label{fig:exit-time}
\end{figure}
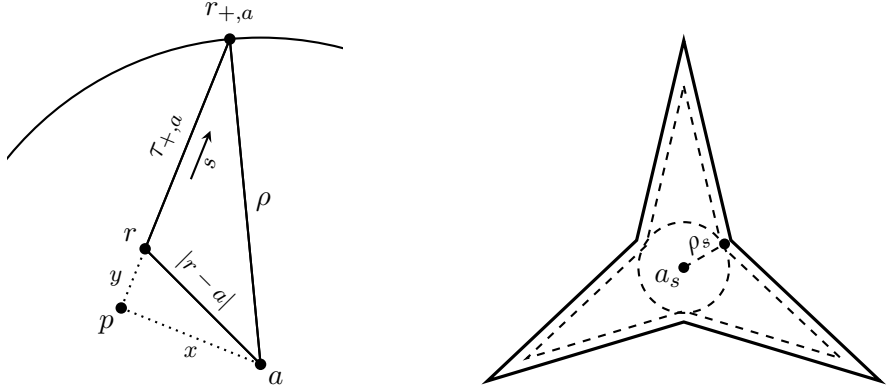
By formal differentiation of this formula and the squeezing theorem, we immediately obtain the Lipschitz estimate 
\begin{align}
|\tau_{+,a}(r,s) - \tau_{+,a}(r,s')| \le 2 \rho \, |s - s'| \quad \forall s,s' \in \Sd, \ r \in B(a,\rho).
\end{align}
By the triangle inequality, we further conclude that $\tau_{+,a}(r,s)\le\tau_{+,a}(r,s')+2\rho \,|s-s'|$ uniformly.
For a general $\rho$-convex domain $\R$, we deduce from \eqref{eq:r-convex} that
\begin{align}
\tau_+(r,s) = \inf_{a \in \A} \tau_{+,a}(r,s) \quad \forall r \in \R, \ s \in \Sd.
\end{align}
Note that $r \in B(a,\rho)$ for all $a \in \A$. By taking the infimum over the previous inequalities for balls $B(a,\rho)$, we thus obtain $\tau_+(r,s)\le\tau_+(r,s')+2\rho \,|s-s'|$. 
Interchanging $s$ and $s'$ yields the reverse inequality,
which proves the Lipschitz estimate \eqref{eq:lipschitz} as claimed.
\end{proof}

The smoothness of the exit times with respect to the direction $s$ allows us to establish trace estimates using characteristics, similar to the radiative transfer equation~\cite{Agoshkov1998, CessenatDautray1984}.  
\begin{lemma}\label{lem:trace_bound}
Let $\R \subset \Rd$, $d \ge 2$ be a $\rho$-convex domain. Then for all $v\in C^1(\overline\R\times\Sd)$
\begin{align} \label{eq:trace2}
    \int_{\partial \R \times \Sd} |v(r,s)|^2\,|s\cdot n(r)|\,\tau_\pm(r,s) \,\dd(r,s) 
    \le (1+4\rho) \, \|v\|_{\hyp}^2.
\end{align}
\end{lemma}
\begin{proof}
Fix $r \in \partial \R$ and $s \in \Sd$, and assume $s \cdot n(r)<0$. In this case, $\tau_+(r,s)>0$ and we write $r_+=r + \tau_+ s$ for the exit point of the ray $r + t s$; see Figure~\ref{fig:exit-time}.
For given $v \in C^1(\overline \R \times \Sd)$, the function 
$\varphi(t)=v(r+ts,s)^2\,\tau_+(r+ts,s)$ is continuously differentiable with
\begin{align*}
\varphi'(t) 
= 2\,v(r+ts,s)\,(s\cdot\nabla_r v(r+ts,s))\,\tau_+(r+ts,s) - v(r+ts,s)^2.
\end{align*}
For the second term, we used that $\frac{d}{dt} \tau_+(r+ts,s)=-1$, since $|s|=1$ and the travel time decreases along the ray $r + t s$. 
Further note that $\tau_+(r_+,s)=0$, since the ray leaves $\overline \R$ at $r_+$. 
Therefore, $\varphi(\tau_+(r,s))=0$, and the fundamental theorem of calculus gives
\begin{align*}
&v(r,s)^2\,\tau_+(r,s)
= \varphi(0) = \varphi(0) - \varphi(\tau_+) 
    = -\int_0^{\tau_+(r,s)} \varphi'(t) \, \mathrm{d}t \\
&=
    \int_0^{\tau_+(r,s)}v(r+ts,s)^2 \,
    - 2 \, v(r+ts,s)\,(s\cdot\nabla_r v(r+ts,s))\,\tau_+(r+ts,s)\,\mathrm{d}t.
\end{align*}
We write $\Gamma_{in}=\{(r,s) \in \partial\R \times \Sd : s \cdot n(r)<0\}$ for the \emph{inflow part} of the phase space boundary, and recall that $\int_{\Gamma_{in}}\int_0^{\tau} f(r+ts) |s\cdot n| \,\mathrm d t\,\mathrm d(r,s) = \int_{\R\times\Sd} f(r,s)\,\mathrm d(r,s)$ for any $f\in L^1(\R\times\Sd)$; see e.g. \cite[Lem.~1]{ChoulliStefanov99}. 
Since $\tau_+(r,s)=0$ on the remaining part of the boundary, we obtain 
\begin{align*}
\int_{\partial\R \times \Sd} &|v(r,s)|^2\,|s\cdot n(r)|\,\tau_+(r,s)\,\mathrm{d}(r,s) 
= \int_{\Gamma_{in}}|v(r,s)|^2\,|s\cdot n(r)|\,\tau_+(r,s)\,\mathrm d(r,s) \\
&=\ \int_{\R\times\Sd} v(r,s)^2 \, - \,  2 \; v(r,s)\,(s\cdot\nabla_r v(r,s))\,\tau_+(r,s)\,\mathrm d(r,s) \\
   & \leq \|v\|_{L^2(\R\times\Sd)}^2 +  2\,\|s\cdot\nabla_r v\|_{L^2(\R;H^{-1}(\Sd))}\,\|v\tau_+\|_{L^2(\R;H^1(\Sd))}. 
\end{align*}
Together with Lemma~\ref{lem:tau-p}, we conclude that
$\|v \tau_+\|_{L^2(\R,H^1(\Sd))} \le 2\sqrt{3} \rho \|v\|_{L^2(\R,H^1(\Sd))}$ which yields the claimed estimate for the weight $\tau_+$ after applying Young's inequality for the product term. The opposite weight $\tau_-$ is treated similarly. 
\end{proof}

\paragraph{Proof of Theorem~\ref{thm:trace}.}
Since $\tau(r,s) = \tau_+(r,s) + \tau_-(r,s)$, the estimate \eqref{eq:trace} now follows immediately by adding up the bounds of Lemma~\ref{lem:trace_bound} for the two weights $\tau_+$ and $\tau_-$.
Since every $\rho$-convex domain is bounded and has a Lipschitz boundary, the second assertion then follows from Theorem~\ref{thm:density}
and the extension-by-continuity principle; see e.g.~\cite[E5.3]{Alt2016}. \qed

\bigskip

\begin{remark}[The weights $\tau$ and $|s\cdot n|$ are not equivalent] $ $\\
With the geometric argument used in Lemma~\ref{lem:tau-p}, one can show that 
\begin{align} 
\tau_{+,a}(r,s) = 2\rho\,|s\cdot n(r)| \qquad  \text{for all } (r,s)\in\Gammain \subset B(a,\rho).
\end{align}
This follows immediately from the explicit formula for the exit time and noting that $s\cdot n(r)<0$ on $\Gammain$. 
For a general $\rho$-convex domain $\R$ and $r \in \partial\R$, there exists an active ball $B(a,\rho)$ with $r\in\partial B(a,\rho)$ and hence $n(r)=(r-a)/\rho$. 
Since  $\tau_+(r,s)\leq \tau_{+,a}(r,s)$, we immediately obtain
\begin{align}
   \tau_+(r,s)\leq 2\rho \, |s\cdot n(r)| \qquad \text{for a.e. } (r,s)\in\Gammain.
    \label{eq:sn-lower}
\end{align}
This shows that $\tau(r,s) |s \cdot n(r)| \le 2\rho |s \cdot n(r)|^p$ for all $1 \le p \le 2$.
The reverse bound does not hold for all $\rho$-convex domains. 
Hence the trace norm in our estimate \eqref{eq:trace} is, in general, weaker than the norms arising in \eqref{eq:trace_nv} for $p \le 2$.
The questions, if trace estimates of the form \eqref{eq:trace_nv} remain valid for $\rho$-convex domains or other special classes of non-smooth domains, seems open.
\end{remark}

\section{Density of smooth functions}
\label{sec:density}

As an intermediate step, we will establish the density of smooth functions in $\Hyp$ for domains $\R$ that are \emph{star-shaped with respect to a ball $B(a_s,\rho_s)$}, i.e., 
\begin{align} \label{eq:star-shaped}
\text{for every $r\in\R$, the convex hull of $\{r\}\cup B(a_s,\rho_s)$ is contained in $\R$.}
\end{align}
%\end{definition}
%
Every bounded open convex set is star-shaped with respect to \emph{every} ball
$B(a_s,\rho_s)\subset\R$.
In particular, every $\rho$-convex domain is
star-shaped with respect to each ball it contains.
%
% \begin{figure}[ht!]
% \centering
% \input{pics/star-shaped.tikz}
% \caption{Typical examples of star-shaped domains. 
% }
% \end{figure}
%
The following theorem generalizes well-known results for domains with smooth boundaries. 
\begin{proposition}[Density in $\Hyp$]
\label{pro:density} $ $\\
Let $\R\subset\RR^d$ be a bounded domain and star-shaped with respect to some ball
$B(a_s,\rho_s)\subset\R$ with radius $\rho_s>0$.
Then
$C^\infty(\overline\R\times\Sd)$ is dense in $\Hyp$.
\end{proposition}

To streamline the proof of this result, we first present two auxiliary lemmas whose assertions are quite standard. For any given $\lambda \in (0,1)$ we define the global dilation map 
\begin{align} \label{eq:dilation}
T_\lambda(r)=a_s+\lambda(r-a_s), \quad r\in\Rd.
\end{align}
As a first step, we show that $T_\lambda(\overline{\R})$ is compactly included in $\R$ which allows us to apply usual mollification arguments; cf. Figure~\ref{fig:exit-time} for a sketch.
\begin{lemma}[Compact inclusion after dilation]
\label{lem:dilation} $ $\\
Let $\R\subset\mathbb{R}^d$ be star-shaped with respect to the ball
$B(a_s,\rho_s)$ and $\lambda\in(0,1)$. Then
\begin{align}
B\big(T_\lambda(r),(1-\lambda)\rho_s\big)
\subset\R
\qquad\text{for every }r\in\overline\R.
\label{eq:dilation1}
\end{align}
As a consequence, we have $T_\lambda(\overline\R)\subset\R$ and $
\dist\bigl(T_\lambda(\overline\R),\partial\R\bigr)
\ge (1-\lambda)\rho_s$.
\end{lemma}
\begin{proof}
We first consider $r\in\R$ and let $|z| < (1-\lambda)\rho_s$. Then we can expand
\begin{align*}
T_\lambda(r)+z 
= \lambda r + (1-\lambda)
\big(a_s+\tfrac{z}{1-\lambda}\big).
\end{align*}
Since $\big|\frac{z}{1-\lambda}\big|<\rho_s$,
we see that $a_s+\frac{z}{1-\lambda}\in B(a_s,\rho_s)$.
Thus $T_\lambda(r)+z$ is a convex combination of $r$ and a
point in $B(a_s,\rho_s)$. Since $\R$ is star-shaped with respect
to $B(a_s,\rho_s)$, this already implies the inclusion \eqref{eq:dilation1} for $r \in \R$. 
By continuity of $T_\lambda(\cdot)$ and $\lambda \in (0,1)$, the inclusion remains valid for all $r \in \overline \R$. 
The remaining claims now follow immediately.
\end{proof}

As a second technical tool, we adapt the classical proof of continuity of translation in $L^p$, see e.g.
\cite[Thm.~4.15]{Alt2016}, to the dilation $T_\lambda$ introduced above.
\begin{lemma}[Continuity under dilation]\label{lem:dilation2} $ $\\
Let $X$ be some Banach space.
Then, for every $f\in L^2(\R;X)$, we have 
\begin{align} \label{eq:dilation2}
f\circ T_\lambda \to f \quad \text{in } L^2(\R;X)  \quad \text{as } \lambda\to 1.
\end{align}
\end{lemma}
\begin{proof}
%% Step~1:
Let $C_\lambda f=f\circ T_\lambda$. Then for all $\lambda\in[\tfrac{1}{2},1]$, we see that
\begin{align} \label{eq:step1}
    \|C_\lambda f\|_{L^2(\R;X)}^2
    \ =\
    \int_\R\|f(T_\lambda(r))\|_X^2\,\mathrm dr
    \ =\
    \lambda^{-d}\int_{T_\lambda(\R)}\|f(r')\|_X^2\,\mathrm dr'
    \ \le\
    2^{d}\,\|f\|_{L^2(\R;X)}^2.
\end{align}
In the second step, we used the transformation formula for integrals and the fact that $T_\lambda$ is differentiable with $\operatorname{det}(D T_\lambda(r)) = \lambda^d$, and in the third that $\lambda \ge 1/2$. 
This shows that the composition maps $C_\lambda$ are uniformly bounded on $L^2(\R;X)$ for $\lambda$ close to one. 

%% Step~2:
From \eqref{eq:dilation}, we get
$\sup_{r\in\overline\R}|T_\lambda(r)-r| \leq |\lambda-1|\operatorname{diam}(\R)$. 
Since any $\phi\in C(\overline\R;X)$ is uniformly continuous on the compact set $\overline\R$,
we can find for each $\epsilon>0$ a $\lambda_\phi(\epsilon)\in(0,1)$ such that
\begin{align} \label{eq:step2}
    \|C_\lambda\phi-\phi\|_{L^2(\R;X)} 
    < \epsilon  \qquad\text{for all } \lambda \in [\lambda_\phi(\epsilon),1].
\end{align}
%
%% Step~3:
For given $f\in L^2(\R;X)$ and $\epsilon>0$, we may choose $\phi\in C(\overline\R;X)$ with
$\|f-\phi\|_{L^2(\R;X)}<\epsilon$; this is possible because $C(\overline\R;X)$ is dense in $L^2(\R;X)$ for a Banach space $X$; cf. \cite[Thm.~4.15]{Alt2016}.
Then from \eqref{eq:step1} and \eqref{eq:step2}, and triangle inequalities, we further deduce that 
\begin{align*}
    \|C_\lambda f-f\|_{L^2(\R;X)}
    \ &\le\
    \|C_\lambda(f-\phi)\|_{L^2(\R;X)}+\|C_\lambda\phi-\phi\|_{L^2(\R;X)}+\|\phi-f\|_{L^2(\R;X)}
    \\
    &\le\
    2^{d/2}\epsilon+\epsilon+\epsilon.
\end{align*}
Since $\epsilon>0$ was arbitrary, we conclude that $C_\lambda f\to f$ in $L^2(\R;X)$ as $\lambda\to 1$.
\end{proof}

\paragraph{Proof of Proposition~\ref{pro:density}} $ $\\
We can now employ a standard line of arguments.

\smallskip 

\noindent
\emph{Step 1: Dilation.}
Let $v_\lambda = v \circ T_\lambda = C_\lambda v$. 
Applying Lemma~\ref{lem:dilation2} with $X=H^1(\Sd)$ gives $v_\lambda=C_\lambda v\to v$ in $L^2(\R;H^1(\Sd))$. 
The change-of-variables $r'=T_\lambda(r)$ further shows that 
$$ (s\cdot\nabla_r v_\lambda)(r,s)= \lambda\,(s\cdot\nabla_r v)(T_\lambda(r),s)
$$ 
in the sense of distributions. 
Hence, by Lemma~\ref{lem:dilation2} with
$X=H^{-1}(\Sd)$, we see that
\begin{align*}
\|s\cdot\nabla_r v_\lambda&-s\cdot\nabla_rv\|_{L^2(\R;H^{-1}(\Sd))} \\
&\le
\lambda
\|C_\lambda(s\cdot\nabla_rv)-s\cdot\nabla_rv\|_{L^2(\R;H^{-1}(\Sd))}
+(1-\lambda)
\|s\cdot\nabla_rv\|_{L^2(\R;H^{-1}(\Sd))}
\end{align*}
which goes to zero with $\lambda \to 1$. 
Hence we have shown that $
v_\lambda\to v$ in $\Hyp$ with $\lambda \to 1$.

\smallskip 

\noindent 
\emph{Step 2: Mollification in $r$.}
We fix $\lambda \in (0,1)$ and 
define $v_{\lambda,\varepsilon}=\moll_\varepsilon *_r v_\lambda$ for a
standard Euclidean mollifier $\moll_\varepsilon$ on $\mathbb R^d$ acting on the $r$ variable, see \cite[Ch.~4]{Alt2016}; the variable $s\in\Sd$ here acts as a fixed parameter.
By Lemma~\ref{lem:dilation}, we know that $\operatorname{dist}(T_\lambda(\overline\R),\partial\R)\ge(1-\lambda)\rho_s$, so that $v_{\lambda,\varepsilon}$ is well-defined for  $0 < \varepsilon < (1-\lambda) \rho_s$.
Moreover, the convolution in the $r$-variable commutes with distributional differentiation, i.e.
\begin{align*}
s\cdot\nabla_r(\moll_\varepsilon*_r v_\lambda)
=
\moll_\varepsilon*_r(s\cdot\nabla_r v_\lambda) \quad\text{on }\R\times\Sd,
\end{align*}
and it is strongly continuous on
$L^2(\R;H^1(\Sd))$ and on
$L^2(\R;H^{-1}(\Sd))$, since it acts only in the Euclidean variable $r$; see \cite[Thm~4.15(2)]{Alt2016}. 
Since multiplication with $s$ is continuous on $L^2(\R;H^{-1}(\Sd))$, we then obtain with $\varepsilon \to 0$ the convergences
\begin{alignat*}{3}
v_{\lambda,\varepsilon} &\to v_\lambda &\quad&\text{in }L^2(\R;H^{1}(\Sd)),\\
s\cdot\nabla_r v_{\lambda,\varepsilon} &\to s\cdot\nabla_r v_\lambda &&\text{in }L^2(\R;H^{-1}(\Sd)).
\end{alignat*}

\noindent
\emph{Step 3: Spherical smoothing.}
Let $P_\delta=e^{\delta\Delta_s}$, $\delta>0$,
denote the heat semigroup on $\Sd$ and note that $P_\delta$ is 
strongly-continuous and contractive on both $H^1(\Sd)$ and $L^2(\Sd)$. 
We apply $P_\delta$ fibre-wise and conclude with $\delta \to 0$ the convergence of
\begin{align*}
v_{\lambda,\varepsilon,\delta}
:=P_\delta v_{\lambda,\varepsilon} 
\to
v_{\lambda,\varepsilon}
\quad \text{in } L^2(\R;H^1(\Sd)).
\end{align*}
Note that for fixed $\lambda,\varepsilon,\delta>0$, the function 
$v_{\lambda,\varepsilon,\delta}$ is smooth in both variables $r$ and $s$.
Since $P_\delta$ and $\nabla_r$ act on different variables, we further get $
\nabla_rv_{\lambda,\varepsilon,\delta}
=
P_\delta(\nabla_rv_{\lambda,\varepsilon}),
$
and 
\begin{align*}
\nabla_rv_{\lambda,\varepsilon,\delta}
\to
\nabla_rv_{\lambda,\varepsilon}
\quad\text{in }L^2(\R\times\Sd).
\end{align*}
Since multiplication by $s$ is bounded from
$L^2(\Sd;\mathbb R^d)$ to $H^{-1}(\Sd)$, this immediately implies $
s\cdot\nabla_rv_{\lambda,\varepsilon,\delta}
\to
s\cdot\nabla_rv_{\lambda,\varepsilon}$
in $L^2(\R;H^{-1}(\Sd))$ and, hence, 
$v_{\lambda,\varepsilon,\delta}
\to
v_{\lambda,\varepsilon}$ in $\Hyp$.

\smallskip 
\noindent 
\emph{Step 4: Diagonal sequence.}
Choosing $\lambda_n \nearrow 1$, $\varepsilon_n=\varepsilon(\lambda_n)\searrow 0$ and $\delta_n=\delta(\lambda_n,\varepsilon_n)\searrow0$ appropriately gives a diagonal sequence $v_{\lambda_n,\varepsilon_n,\delta_n}\in
C^\infty(\overline\R\times\Sd)$ converging to $v$ in $\Hyp$.\qed

\paragraph{Proof of Theorem~\ref{thm:density}} $ $\\
We employ the localization argument used in the proof of \cite[Prop.~2.2]{AlbrittonEtAl24}.  
Every bounded Lipschitz domain is locally star-shaped with respect to some balls, i.e., there exists a finite open covering $\R \subset \bigcup_k B(r_k,\rho_k')$ such that each subdomain $\R_k = \R \cap B(r_k,\rho_k')$ is star-shaped with respect to some ball $B(a_k,\rho_k)$.
We define $v_k=v|_{\R_k}$ and see that $v_k$ can be approximated by smooth functions $\tilde v_k$ using Proposition~\ref{pro:density}.
We further denote by $1=\sum_k \chi_k$ a partition of unity, i.e., a set of functions $\chi_k \in C^\infty$ with $0 \le \chi_k \le 1$ and $\text{supp}(\chi_k) \subset B(r_k,\rho_k')$. 
Then $\tilde v = \sum_k \chi_k \tilde v_k$ is smooth and approximates $v = \sum_k \chi_k v_k$ as desired. \qed

\section{Possible extensions}
\label{sec:discussion}

Following the arguments of \cite[Prop.~6.1]{AlbrittonEtAl24} and \cite[Prop.~3.1]{BrunkenSmetana22}, the density result stated in Theorem~\ref{thm:density} can be generalized to the time dependent kinetic Sobolev spaces 
\begin{align*}
H^1_{kin}
=
\left\{
v\in L^2(I\times\R;H^1(\Sd)):
(\partial_t+s\cdot\nabla_r)v
\in L^2(I\times\R;H^{-1}(\Sd))
\right\},
\end{align*}
where $I=(0,T)$ denotes a time interval.
By integration over space-time characteristics $(t+\theta,r+\theta s)$, $\theta\in\mathbb R$, the existence of traces on $\partial(I\times\R)\times\Sd$ for functions in $H^1_{kin}$ and a bound corresponding to \eqref{eq:trace2} can be established with similar arguments as above. To that end, we note that the exit times $\tilde\tau_\pm(t,r,s)$ for the space-time cylinder $I\times \R$ satisfy
$$
\tilde\tau_+(t,r,s)=\min(\tau_+(r,s),T-t),\qquad 
\tilde\tau_-(t,r,s)=\min(\tau_-(r,s),t).
$$
Therefore,
$\tilde\tau_\pm(t+\theta,r+\theta s,s)=\tilde\tau_\pm(t,r,s)\mp\theta$
and
$\tilde\tau_\pm(t,r,s)$ are Lipschitz-continuous functions of $s$ with Lipschitz constant $2\rho$ for fixed $t$ and $r$ and $\rho$-convex domain $\R$, by Lemma~\ref{lem:tau-p}.

A more challenging question concerns the extension of the trace estimates to domains which are not uniformly convex or do not possess a smooth boundary. 
As noted in \cite{NV26}, the Lipschitz estimate of Lemma~\ref{lem:tau-p} does in general not hold in such cases. 
In a forth-coming publication, we will use different methods of proof to establish quantitative trace estimates 
\begin{align}\label{eq:trace3}
\int_{\partial\R \times \Sd} |v(r,s)|^2 \, w(r,s) \, |s\cdot n(r)| \, \dd(r,s) 
\leq C \, \|v\|_{\hyp}^2.
\end{align}
with appropriate weight functions $w$ for polyhedral domains $\R$ in dimension $d=2$ and $3$.

\bibliographystyle{plain}
\bibliography{references_density_trace}
\end{document}

%% file: r-convex.tikz
\begin{tikzpicture}[font=\normalsize,scale=0.8]

% ============================================================
% LEFT: disk
% ============================================================

\pgfmathsetmacro{\R}{2}
\pgfmathsetmacro{\r}{0.07}

\begin{scope}[xshift=0cm,yshift=0cm,scale=1.35]

  % center of the disk
  \coordinate (a) at (0,0);
  \fill (a) circle (0.07);
  \node[below left] at (a) {$a$};

  % disk
  \draw[very thick] (a) circle (2);

  % radius rho
  %\coordinate (b) at (1.4142,-1.4142);
  \coordinate (b) at (1.7321,-1);
  \fill (b) circle (0.07);
  \draw[thick,dashed] (a) -- (b)  node[pos=0.55,above] {$\rho$};

  \coordinate (c) at (-2,0); 
  \fill (c) circle (0.07);
  \node[left] at (c) {$r_-$};

  \coordinate (d) at (1.4142,1.4142); 
  \fill (d) circle (0.07);
  \node[right] at (d) {$r_+$};
 
  %\coordinate (r) at (-0.9757,0.4243);
  \coordinate (r) at (-0.6343,0.5657);
  \fill (r) circle (0.07);
  \node[below right] at (r) {$r$};

  \draw[thick] (c) -- (r) node[pos=0.55,above=-2pt,rotate=22,font=\scriptsize] {$\tau_-(r,s)$};
  \draw[thick] (r) -- (d) node[pos=0.5,above=-2pt,rotate=22,font=\scriptsize] {$\tau_+(r,s)$};

  \draw[thick,->] (0.1192,0.6778) -- (0.8021,0.9606) node[pos=0.55,below] {$s$};

\end{scope}

% ============================================================
% RIGHT: Reuleaux triangle
% ============================================================

\begin{scope}[xshift=5.5cm,yshift=-2cm,scale=1.35]

  \pgfmathsetmacro{\side}{3.8}

  \coordinate (V1) at (0,0);
  \fill (V1) circle (0.07);
  \node[below left] at (V1) {$a_1$};

  \coordinate (V2) at (\side,0);
  \fill (V2) circle (0.07);
  \node[below right] at (V2) {$a_2$};

  \coordinate (V3) at (\side/2,{sqrt(3)/2*\side});
  \fill (V3) circle (0.07);
  \node[above] at (V3) {$a_3$};

  % Reuleaux triangle
  \draw[very thick]
    ($(V1)+(0:\side)$) arc (0:60:\side);

  \draw[very thick]
    ($(V2)+(120:\side)$) arc (120:180:\side);

  \draw[very thick]
    ($(V3)+(240:\side)$) arc (240:300:\side);

  % radius rho
  \coordinate (M) at ($(V1)+(30:\side)$);
  \fill (M) circle (0.07);

  \draw[thick,dashed] (V1) -- (M)
    node[pos=0.55,above left] {$\rho$};

  \draw[thick,dotted] (V1) -- (V2) -- (V3) -- (V1);
%    node[pos=0.55,above left,font=\scriptsize] {$\rho$};

  % % caption
  % \node[below=0.35cm,font=\footnotesize]
  %   at (1,0) {Reuleaux triangle of width $\rho$};

\end{scope}

\end{tikzpicture}

%% file: exit-time.tikz
\begin{tikzpicture}[scale=1.6,>=stealth]
\begin{scope}[xshift=0cm,yshift=1cm,scale=1.35]
  % parameters (names chosen to avoid clobbering \rho, \tau, ...)
  \pgfmathsetmacro{\Rad}{2}
  \pgfmathsetmacro{\rlen}{\Rad/2}
  \pgfmathsetmacro{\rang}{135}
  \pgfmathsetmacro{\sang}{68}
  \pgfmathsetmacro{\Rx}{\rlen*cos(\rang)}
  \pgfmathsetmacro{\Ry}{\rlen*sin(\rang)}
  \pgfmathsetmacro{\sxv}{cos(\sang)}
  \pgfmathsetmacro{\syv}{sin(\sang)}
  \pgfmathsetmacro{\dotp}{\Rx*\sxv+\Ry*\syv}
  \pgfmathsetmacro{\hsq}{\Rx*\Rx+\Ry*\Ry-\dotp*\dotp}
  \pgfmathsetmacro{\sqterm}{sqrt(\Rad*\Rad-\hsq)}
  \pgfmathsetmacro{\tauv}{\sqterm-\dotp}
  \pgfmathsetmacro{\Ex}{\Rx+\tauv*\sxv}
  \pgfmathsetmacro{\Ey}{\Ry+\tauv*\syv}
  \pgfmathsetmacro{\Px}{\Rx-\dotp*\sxv}
  \pgfmathsetmacro{\Py}{\Ry-\dotp*\syv}

  \coordinate (A) at (0,0);
  \coordinate (R) at (\Rx,\Ry);
  \coordinate (Exit) at (\Ex,\Ey);
  \coordinate (P) at (\Px,\Py);

  % circle, clipped to the region of interest (avoids a mostly-empty full disk)
  \begin{scope}
    \clip (-1.55,-0.15) rectangle (0.5,2.28);
    \draw[thick] (A) circle (\Rad);
  \end{scope}
  % \useasboundingbox (-1.55,-0.15) rectangle (0.5,2.28);

  % completion (dotted)
  \draw[dotted,thick] (A) -- (P) node[pos=0.5,below,font=\footnotesize] {$x$};
  \draw[dotted,thick] (P) -- (R) node[pos=0.5,left,font=\footnotesize] {$y$};

  % original triangle (solid)
  \draw[thick] (R) -- (A);
  \draw[thick] (A) -- (Exit);
  \draw[thick] (R) -- (Exit);

  % points
  \fill (A) circle (0.035) node[below right=-1pt] {$a$};
  \fill (R) circle (0.035) node[above left=-1pt] {$r$};
  \fill (Exit) circle (0.035) node[above=2pt] {$r_{+,a}$};
  \fill (P) circle (0.035) node[below left=-1pt] {$p$};

  % separate direction arrow for s, offset to the top left of the tau_a edge
  \draw[->,thick] ($(R)+(\sang:0.5)+(\sang+90:-0.1)$) -- ++(\sang:0.32)
    node[below left=5pt,rotate=80,font=\footnotesize] {$s$};

  % labels: 
  \draw[thick] (R) -- (A) node[pos=0.4,sloped,above=-1pt,font=\footnotesize] {$|r-a|$};
  \draw[thick] (R) -- (Exit) node[pos=0.5,above,sloped] {$\tau_{+,a}$};
  \draw[thick] (A) -- (Exit) node[pos=0.5,right] {$\rho$};

  % \node[rotate=\sang] at ($(P)+(\sang:-0.15)+(\sang+90:0.05)$)    {$x$};
  % \node[rotate=\sang,font=\scriptsize] at ($(P)+(\sang:-0.15)+(\sang+90:0.24)$)
  %   {$|\langle r-a,s\rangle|$};

% ============================================================
\end{scope}
\begin{scope}[xshift=3.5cm,yshift=1.8cm,scale=0.75]

  % ------------------------------
  % six vertices: three outer and three inner
  % ------------------------------

  \pgfmathsetmacro{\Rout}{2.5}
  \pgfmathsetmacro{\Rin}{0.6}
  \pgfmathsetmacro{\rhoS}{0.5}

  % dilation/contraction parameter
  \pgfmathsetmacro{\lam}{0.8}

  \coordinate (S1) at (90:\Rout);
  \coordinate (S2) at (30:\Rin);
  \coordinate (S3) at (-30:\Rout);
  \coordinate (S4) at (-90:\Rin);
  \coordinate (S5) at (-150:\Rout);
  \coordinate (S6) at (150:\Rin);

  % --------------------------------------
  % center of the star-shaped ball
  % --------------------------------------

  \coordinate (as2) at (0,0);

  % --------------------------------------
  % original star-shaped domain
  % --------------------------------------

  \draw[very thick]
    (S1) -- (S2) -- (S3) -- (S4) -- (S5) -- (S6) -- cycle;

  % --------------------------------------
  % dilated star-shaped domain
  % r -> a_s + lambda (r-a_s)
  % --------------------------------------

  \coordinate (S1d) at ($(as2)!\lam!(S1)$);
  \coordinate (S2d) at ($(as2)!\lam!(S2)$);
  \coordinate (S3d) at ($(as2)!\lam!(S3)$);
  \coordinate (S4d) at ($(as2)!\lam!(S4)$);
  \coordinate (S5d) at ($(as2)!\lam!(S5)$);
  \coordinate (S6d) at ($(as2)!\lam!(S6)$);

  \draw[thick,dashed]
    (S1d) -- (S2d) -- (S3d) -- (S4d) -- (S5d) -- (S6d) -- cycle;

  % --------------------------------------
  % center of the star-shaped ball
  % --------------------------------------

  \fill (as2) circle (0.06);
  \node[below left=-3pt] at (as2) {$a_s$};

  % ball
  \draw[thick,dashed]
    (as2) circle (\rhoS);

  % radius
  \draw[thick,dashed]
    (as2) -- ++(30:\rhoS)
    node[pos=1,above left=-2pt,font=\small,rotate=25] {$\rho_s$};

  \fill (0.45,0.26) circle (0.06);

\end{scope}
  
\end{tikzpicture}